\documentclass[11pt, reqno]{amsart}
\usepackage{latexsym,amsmath,amssymb, amsthm}
\usepackage{cases, verbatim, esint, todonotes}
\usepackage[active]{srcltx}
\usepackage[colorlinks]{hyperref}
\usepackage{hypernat}
\usepackage{xcolor}
\usepackage{lineno}

\usepackage[margin=1.5in]{geometry}

\newtheorem{thm}{Theorem}
\newtheorem{lem}[thm]{Lemma}
\newtheorem{prop}[thm]{Proposition}

\theoremstyle{remark}
\newtheorem{rmk}[thm]{Remark}
\newtheorem{question}{Question}
\newtheorem{example}[thm]{Example}

\theoremstyle{definition}
\newtheorem{defi}[thm]{Definition}

\numberwithin{thm}{section} 
\numberwithin{equation}{section}

\makeatletter

\newcommand{\Rmnum}[1]{\expandafter\@slowromancap\romannumeral #1@}
\makeatother

\def\polhk#1{\setbox0=\hbox{#1}{\ooalign{\hidewidth
    \lower1.5ex\hbox{`}\hidewidth\crcr\unhbox0}}}

\def\Z{\mathbb{Z}}
\def\R{\mathbb{R}}
\def\Rp{\mathbb{R}_+}

\def\T{\mathbb{T}}
\def\H{\mathbb{H}}
\def\L{\mathbb{L}}
\def\X{X}

\def\Colm{\underline{\mathcal{C}}}

\def\R{{\mathbb R}}
\def\X{{\mathbf X}}

\def\H{{\mathbb H}}

\def\S{{\mathbb S}}

\def\T{{\mathbb T}}
\def\L{{\mathbb L}}

\newcommand{\vep}{\varepsilon}

\newcommand{\bpm}{\begin{pmatrix}}
\newcommand{\epm}{\end{pmatrix}}

plus 6pt minus 12pt
\title[Convexity preserving in geodesic spaces]{\protect{Convexity preserving properties for Hamilton-Jacobi equations in geodesic spaces}}

\author[Q. Liu]{Qing Liu}
\address{Qing Liu, Department of Applied Mathematics, Faculty of Science, Fukuoka University, Fukuoka 814-0180, Japan, {\tt qingliu@fukuoka-u.ac.jp}}

\author[A. Nakayasu]{Atsushi Nakayasu}
\address{Atsushi Nakayasu, Graduate School of Mathematical Sciences, University of Tokyo, 3-8-1 Komaba, Meguro-ku, Tokyo, Japan, {\tt ankys@ms.u-tokyo.ac.jp}}

\date{\today}

\begin{document}

\begin{abstract}
We study the convexity preserving property for a class of time-dependent Hamilton-Jacobi equations in a complete geodesic space. Assuming that the Hamiltonian is nondecreasing, we show that in a Busemann space the unique metric viscosity solution preserves the geodesic convexity of the initial value at any time. We provide two approaches and also discuss several generalizations for more general geodesic spaces including the lattice graph. 
\end{abstract}


\maketitle

\section{Introduction}
Inspired by a recent trend in the study of fully nonlinear partial differential equations in metric spaces, in this paper we discuss the convexity preserving property for a class of first-order Hamilton-Jacobi equations in metric spaces. More precisely, for a metric space $\X$ equipped with a metric $d$, we consider the Hamilton-Jacobi equation
\begin{linenomath}
\begin{numcases}{\textrm{(HJ)}\quad}
\partial_t u+H(|\nabla u|)=0 &\text{in $\X\times (0, \infty)$,} \label{hj1} \\
u(x, 0)=u_0(x) &\text{in $\X$,}  \label{hj2}
\end{numcases}
\end{linenomath}
and show, under appropriate assumptions on the space $(\X, d)$ and the Hamiltonian $H$ defined on $\Rp := [0, \infty)$, that the unique continuous viscosity solution $u(x, t)$ is convex in the space variable $x$ at any time $t\geq 0$ provided that the initial value $u_0$ is convex in $\X$.
The notion of convexity of functions here on metric spaces will be clarified later. 

Well-posedness of the Hamilton-Jacobi equation above in the framework of viscosity solutions has recently been established in a large class of metric spaces called \emph{geodesic spaces} \cite{AF, GaS2, GHN, GaS, Na1, NN}; see also well-posedness results and applications on networks \cite{Schi, ACCT, IMZ, SC, IMo1, IMo2, CCM}. A metric space $(\X, d)$ is said to be \emph{geodesic} if for any $x, y\in \X$, there exists a geodesic $\gamma_t$ ($t\in [0, 1]$) in $\X$ joining $x, y$ with a constant speed; in other words, we have $\gamma_0=x$, $\gamma_1=y$ and $d(\gamma_s, \gamma_t)=|s-t|d(x, y)$ for any $s, t\in [0, 1]$. Throughout this paper, we assume that $(\X, d)$ is a complete geodesic space. 
We also understand the term $|\nabla u|$ as the \emph{local slope}
\[
|\nabla u|(x, t) := \limsup_{y \to x} \frac{|u(y, t)-u(x, t)|}{d(x, y)}
\]
if the right hand side is finite.

\subsection{Background and motivations}
Convexity properties are well known for general parabolic or elliptic equations in the Euclidean spaces. The convexity of classical solutions to uniformly elliptic equations were studied by Korevaar \cite{Kor} and Kennington \cite{Ke} by establishing a convexity maximum principle. Such a convexity maximum principle was then generalized in the framework of viscosity solutions to a very general class of degenerate parabolic equations in \cite{GGIS}. We remark that the convexity of solutions to various PDEs has also been extensively studied in \cite{K2, K3, Sak, DK, ALL, Ju} etc. 

As for the mean curvature flow and other geometric motions, we refer to \cite{Hui} for a well-known result on the convexity preserving property of an evolving surface; see also \cite{GH} for the two dimensional case in detail. This property was later formulated in terms of level set method in \cite{ES1} and \cite{GGIS}, and was recently proved in \cite{LSZ} by using a deterministic game-theoretic approach established by Kohn and Serfaty \cite{KS1}. 

Considering that the framework of viscosity solution theory for the first order equations has been generalized in geodesic spaces, we are interested in extending the convexity results above also to the general circumstances. In general, one cannot expect the convexity preserving property still hold in general geodesic spaces even for very simple equations. In fact, the solution of a first order linear PDE fails to preserve horizontal convexity in the first Heisenberg group, as shown in \cite{LMZ}. It is however not clear whether preservation of horizontal convexity holds even for the Hamilton-Jacobi flow (HJ). In the present work, we consider geodesic spaces with more convexity structure and focus our attention only on the simple equation (HJ) in order to provide a clear view of the ideas. Many of our results can be generalized to handle the Hamiltonians that depend also on the space variable $x$ in a concave manner. 

\subsection{Main results}
There are two important questions that need to be answered before proceeding to our main results. First, it is necessary to introduce an appropriate notion of convexity of functions that are defined on metric spaces. Since the space is assumed to be geodesic so as to obtain uniqueness and existence of viscosity solutions,  a natural choice is to employ the so-called (weak) geodesic convexity. More precisely, in Definition \ref{def weak convex} we define a function $f: \X\to \R$ to be weakly geodesically convex in $(\X, d)$ if 
\[
f(x)+f(y)\geq \inf_{z\in M(x, y)} 2f(z) \quad \text{for all $x, y\in \X$.}
\]
Here $M(x, y)$ denotes the set of all midpoints between $x, y\in \X$, which are not necessarily a singleton or even a compact set. It is certainly possible to define strong geodesic convexity by replacing the infimum with supremum in the expression above. We however mainly discuss the weaker notion in this work.

The second issue is related to a proper structure of metric spaces that is compatible with the geodesic convexity. Besides the necessary conditions for convexity preserving property we are concerned with, it is also necessary to avoid the situation that the only convex functions on $\X$ are constants. For example, if $\X=\S^1$ with the metric defined to be the shortest length of arcs connecting  points,  there are no convex functions other than constants. Similar situations also appear when $\X$ is two-dimensional lattice graph $\L^2$ with $l^1$ distance; see Proposition \ref{trivial convex}. 

Attempting to avoid such a trivial situation, we need more assumptions on the metric space. One option is to assume that the metric is convex, i.e., the distance function $d(\cdot , a)$ for every fixed $a\in \X$ is geodesically convex. This turns out to be closely related to the intensive study of spaces with convex metrics; consult \cite{Jbook, Pa} for more details. One of the main objects in the literature is the so-called Busemann space. More precisely, $(\X, d)$ is a Busemann space if for all $x, y, x', y'\in \X$
\begin{equation}
\label{eq busemann ineq 4}
2 d(z, z') \le d(x, x')+d(y, y') \quad \text{for any $z \in M(x, y)$ and $z' \in M(x', y')$.}
\end{equation}
Note that this condition implies that $(\X, d)$ is uniquely geodesic (so that we can denote the only point in $M(x, y) $ by $m(x, y)$ for any $x, y\in \X$) and therefore the notions of weak and strong geodesic convexity are equivalent. Examples of the Busemann space include all Hilbert spaces, hyperbolic spaces, trees, and networks without loops or cycles.
Under this assumption on metric convexity, we obtain the following result.

\begin{thm}[Main result]
\label{thm intro}
Assume that $(\X, d)$ is a complete Busemann space.
Assume that $H \colon \Rp \to \R$ is continuous and non-decreasing.
Let $u$ be the unique continuous viscosity solution of (HJ) with $u_0$ Lipschitz continuous in $\X$.
If $u_0$ is geodesically convex,
then $u(\cdot, t)$ is also geodesically convex for all $t \ge 0$.
\end{thm}

Although such a result looks more or less expected, its proof is less straightforward and relies on the recent development of viscosity solution theory in metric spaces. Our PDE proof is based on an adaptation of the convexity maximum principle to geodesic spaces, for which we utilize the standard technique of doubling variables and employ 
\[
\varphi(x, y, z) =d(z, m(x, y))^2 \quad \text{ for $x, y, z\in \X$}
\] 
as a penalty function. Roughly speaking, on one hand, it is not difficult to find that 
\[
|\nabla_z d(z, m(x, y))|=1
\]
if $z\neq m(x, y)$, and on the other hand, we have
\begin{equation}\label{intro Busemann}
|\nabla_x d(z, m(x, y))|\leq {1\over 2}, \quad |\nabla_y d(z, m(x, y))|\leq {1\over 2},
\end{equation}
owing to the assumption that $(\X, d)$ is a Busemann space. We remark that \eqref{intro Busemann} is essentially a differential version of the condition \eqref{eq busemann ineq 4} and plays a key role in the convexity proof; see Proposition \ref{thm diff mid} for a rigorous derivation. Adapting the estimates above to the notion of metric viscosity solutions introduced in \cite{GaS2, GaS}, we complete our rigorous proof by plugging them into viscosity inequalities and apply the comparison arguments as in \cite{GGIS} etc. 

The monotonicity assumption on $H$ proves to be necessary; see Remark \ref{rmk monotonicity} for a counterexample when $\X=\Rp=[0, \infty)$ and $H$ is decreasing.  Moreover, since the notion of metric viscosity solution in this type of closed region essentially reduces to the state constraint boundary value problem in the Euclidean space (as first introduced by Soner \cite{So1}), our example also shows that viscosity solutions of time-dependent state constraint problems fail to enjoy the convexity preserving property in general, although the convexity of solutions is known to hold in the stationary case \cite{ALL}.

In Theorem \ref{thm convexity1}, motivated by \cite{LSZ}, we present a more direct approach with convex Hamiltonians, which relies on the celebrated Hopf-Lax formula given by \cite{GaS2} in the context of geodesic spaces:
\begin{equation}
\tag{HL}
\label{eq hl}
u(x, t) = \inf_{a \in \X}\left\{ u_0(a)+t L\left({d(a, x) \over t}\right) \right\} \quad \text{for $(x, t) \in \X\times(0, \infty)$,}
\end{equation}
where $L: \Rp\to \R$ denotes the corresponding convex and coercive Lagrangian. 
Given a convex initial value $u_0$ in $\X$, we express the solution explicitly and show that the Hopf-Lax operator never breaks the convexity of $u_0$.
A formal argument is as follows.
Suppose for any endpoints $x, y\in \X$, there exist $a, b\in \X$ such that
\[
u(x, t) \ge u_0(a)+t L\left({d(a, x) \over t}\right),
\quad u(y, t) \ge u_0(b)+t L\left({d(b, y) \over t}\right),
\]
by the Hopf-Lax formula, then due to the fact that $\X$ is a Busemann space and $u_0$ is geodesically convex, the midpoints satisfy
\[
2u_0(m(a, b)) \le u_0(a)+u_0(b),
\quad 2d(m(a, b), m(x, y)) \le d(a, x)+d(b, y).
\]
It immediately follows from the convexity of $L$ that 
\[
2u(m(x, y), t)
\le 2u_0(m(a, b))+2t L\left({d(m(a, b), m(x, y)) \over t}\right)
\le u(x, t)+u(y, t),
\]
as desired.  

An advantage of this method is that it demands weaker convexity of the metric. In Theorem \ref{thm convexity1}, we substitute Busemann's condition \eqref{eq busemann ineq 4} with a local (but uniform) version, which resembles the notion of so-called nonpositively curved spaces in the Busemann space (Busemann NPC space for short). However, in order to use the Hopf-Lax formula, we need to further assume the convexity of the Hamiltonian in addition to its coercivity and monotonicity. The distance function is not convex but there are still examples of geodesically convex functions; see Example \ref{ex npc only} and Example \ref{ex npc only2}.

We continue to investigate the more challenging case when the metric $d$ is no longer convex. We are particularly interested in the discrete spaces such as the two-dimensional lattice graph $\L^2$ with $l^1$ metric. In spite of the fact that $\L^2$ is a Busemann NPC space, Theorem \ref{thm convexity1} reduces to a constant preserving property.  In order to obtain a more meaningful result, we further relax the notion of convexity in $\L^2$. Several attempts are made in Section \ref{sec lattice}, where concrete examples are constructed to reveal the non-preservation of convexity.

For a metric space $\X$ like $\L^2$ and a continuous convex coercive $H: \Rp\to \R$, we are able to show, as elaborated in Theorem \ref{thm convexity3}, that the metric viscosity solution $u$ of 
\begin{linenomath}
\begin{numcases}{\rm (HJ2)\quad}
\partial_t u-H(|\nabla u|)=0 &\text{in $\X\times (0, \infty)$,} \label{eq negative hj} \\
u(\cdot , 0)=u_0 &\text{in $\X$}  \nonumber
\end{numcases}
\end{linenomath}
preserves the following convexity-related property of the initial value $u_0$:
\begin{equation}\label{intro infinity}
2 u_0(z) \le \sup_{B_r(z)} u_0+\inf_{B_r(z)} u_0 \quad \text{for all $z \in \X$ and $r > 0$ small enough,}
\end{equation}
where $B_r(z)$ denote the closed ball centered at $z \in \X$ with radius $r > 0$, that is, $B_r(z) := \{ x \in \X \mid d(z, x) \le r \}$.
A benefit of proposing such a convexity concept lies at the very weak demand on the structure of metric space. In order to understand this type of convexity of metric viscosity solutions, one does not need any assumption more than the existence of geodesics in $\X$.  This notion is also closely related to the game interpretation of infinity Laplacian studied by Peres, Schramm, Sheffield and Wilson \cite{PSSW}. We therefore call the property \eqref{intro infinity} infinity-subharmoniousness, following the terminology in \cite{MPR3}. In the Euclidean space $\R^N$, it amounts to saying that $u_0$ is convex along the direction of its gradient. 

It is worth pointing out that we here  replace the Hamiltonian $H$ by $-H$. As opposed to the argument for Theorem \ref{thm convexity1},  in the proof of Theorem \ref{thm convexity3} we need to apply the Hopf-Lax formula to the midpoint $z$ in the convexity inequality before determining the corresponding endpoints $x$ and $y$. This modified process turns out to be more consistent with a concave Hamiltonian than a convex one. 

\subsection{Organization of the paper}
In Section \ref{sec convex function}, we give several definitions of convex functions in geodesic spaces. In Section \ref{sec convex space}, we discuss the assumptions on the convexity of geodesic spaces, including the properties of Busemann spaces, Busemann NPC spaces etc. Section \ref{sec viscosity} is a review of the known results on the definition of metric viscosity solutions and the Hopf-Lax formula for (HJ). Our main results on the convexity preserving property are presented in Section \ref{sec main}. Section \ref{sec lattice} is devoted to discussions on more general situations such as the lattice $\L^2$ and an application of preservation of infinity subharmoniousness. 

\subsection*{Acknowledgments}
Part of this work was completed while the second author was visiting Fukuoka University, whose support and hospitality 
are gratefully acknowledged. The work of the first author was supported by JSPS Grant-in-Aid for Young Scientists, No. 16K17635. 
The work of the second author was also supported by Grant-in-Aid for JSPS Fellows No. 25-7077.

\section{Convex functions in geodesic spaces}
\label{sec convex function}


Before introducing notions of convexity, we first define midpoints of geodesics in a geodesic space $(\X, d)$.

\begin{defi}[Midpoints]
Let $x$ and $y$ be two points of a complete geodesic space $(\X, d)$.
We say $z \in \X$ is a \emph{midpoint} between $x$ and $y$
if $z$ lies on a geodesic between $x, y$ and satisfies $d(x, z) = d(y, z)$.
We denote by $M(x, y)$ the collection of all midpoints between $x$ and $y$.
\end{defi}

It is clear that
\[
M(x, y) = B_{d(x, y)/2}(x)\cap B_{d(x, y)/2}(y).
\]
In addition, we see for every $z, z' \in M(x, y)$ that
\[
d(z, z') = d(z, x)+d(x, z') \le d(x, y).
\]
Therefore, $M(x, y)$ is a bounded closed set,
and if $\X$ is locally compact, then $M(x, y)$ is compact.

\subsection{Geodesic convexity}

Let us next propose several major definitions of convexity of functions in our present work.

The first one is convexity along geodesics.

\begin{defi}[Geodesic convexity]
\label{def weak convex}
A continuous function $f$ in a complete geodesic space $\X$ is said to be \emph{weakly geodesically convex}
if
\begin{equation}
\label{eq weak convex}
\inf_{z \in M(x, y)}2 f(z) \le f(x)+f(y)
\end{equation}
for any $x, y \in \X$.
\end{defi}

\begin{rmk}
If the midpoint set $M(x, y)$ is compact,
the condition \eqref{eq weak convex} means that there exists some midpoint $z \in M(x, y)$ satisfying
\begin{equation}
\label{eq convex}
2 f(z) \le f(x)+f(y).
\end{equation}
\end{rmk}

Such a notion of convexity also appears in \cite{AGS}.
It can be strengthened by demanding that 
\begin{equation}
\label{eq strong convex}
\sup_{z \in M(x, y)}2 f(z) \le f(x)+f(y)
\end{equation}
instead of \eqref{eq weak convex}.
Any function $f \in C(\X)$ satisfying \eqref{eq strong convex} for any $x, y \in \X$ is said to be \emph{strongly geodesically convex} in this work.

When geodesics between every $x, y \in \X$ are unique,
then the conditions \eqref{eq weak convex} and \eqref{eq strong convex} are equivalent since the midpoint set $M(x, y)$ is singleton.
In this case we say $f$ is \emph{geodesically convex} for short if it satisfies \eqref{eq weak convex} and denote by $m(x, y)$ the unique point in $M(x, y)$.



The above notion of convexity can be localized.

\begin{lem}
\label{local convexity}
A function $f$ on a geodesic space $(\X, d)$ is weakly geodesically convex
if and only if it is locally weakly geodesically convex;
that is, there exists $\delta > 0$ such that \eqref{eq weak convex} holds for all $x, y \in \X$ with $d(x, y) \le \delta$.
\end{lem}


\begin{proof}
It is clear that weak geodesic convexity implies local weak geodesic convexity.
Since $d(x, y) < \infty$ for every $x, y \in \X$,
it suffices to show that any function $f$ that satisfies \eqref{eq weak convex} for any $x, y\in \X$ with $d(x, y)\leq \delta$ also satisfies the same inequality \eqref{eq weak convex} with $d(x, y)\leq 2\delta$.


Let $x, y \in \X$ be such that $d(x, y) \le 2\delta$ and fix $z_0 \in M(x, y)$.
Note that $d(x, z_0) = d(z_0, y) = d(x, y)/2 \le \delta$.
Then we may apply the convexity assumption \eqref{eq weak convex} of $f$ to obtain
\begin{equation}
\label{eq:loc convex1}
 \inf_{z_1 \in M(x, z_0)} 2f(z_1) \le f(x)+f(z_0),
\qquad  \inf_{z_2 \in M(z_0, y)} 2f(z_2) \le f(z_0)+f(y).
\end{equation}

Let us take arbitrary $z_1 \in M(x, z_0)$ and $z_2\in M(z_0, y)$.
We first claim 
\begin{equation}
\label{eq:loc convex3}
d(z_1, z_2) = {1\over 2}d(x, y).
\end{equation}
Indeed, we easily get
\[
d(z_1, z_2) \le d(z_1, z_0)+d(z_0, z_2) = {1\over 2}d(x, y)
\]
and, on the other hand, 
\[
d(x, y) \le d(x, z_1)+d(z_1, z_2)+d(z_2, y) = {1\over 2}d(x, y)+d(z_1, z_2).
\]
Then \eqref{eq:loc convex3} follows immediately.
Since $d(z_1, z_2) \le \delta$,
using again \eqref{eq weak convex} with $x=z_1$ and $y=z_2$,
we are led to
\begin{equation}
\label{eq:loc convex2}
\inf_{z \in M(z_1, z_2)} 2f(z) \le f(z_1)+f(z_2).
\end{equation}

We next show that 
\begin{equation}\label{eq:loc convex5}
M(z_1, z_2) \subset M(x, y).
\end{equation}
For any fixed $z_\vep \in M(z_1, z_2)$,
we can easily see by \eqref{eq:loc convex3} that 
\[
d(z_1, z_\vep) = d(z_\vep, z_2) = {1 \over 4}d(x, y),
\]
and therefore
\begin{equation}
\label{eq:loc convex4}
d(x, z_\vep) \le d(x, z_1)+d(z_1, z_\vep) = {1\over 2}d(x, y),
\quad d(z_\vep, y) \le d(z_\vep, z_2)+d(z_2, y) = {1\over 2}d(x, y).
\end{equation}
Since $d(x, y) \le d(x, z_\vep)+d(z_\vep, y)$, the equalities in \eqref{eq:loc convex4} hold,
and so $z_\vep \in M(x, y)$.

Combining \eqref{eq:loc convex1}, \eqref{eq:loc convex2} and \eqref{eq:loc convex5}, we end up with 
\[
\inf_{z \in M(x, y)} 4f(z) \le f(x)+f(y)+2 f(z_0).
\]
Since $z_0 \in M(x, y)$ is arbitrary, we conclude the proof.
\end{proof}

\subsection{Pointwise convexity}

For our applications in some particular metric spaces,
we also define a different type of convexity, which is much weaker than the geodesic convexity but relies far less on the geometry of the metric space.

\begin{defi}[Infinity-subharmoniousness]
\label{def wp-convex}
We say a continuous function $f$ in a metric space $(\X, d)$ is \emph{$\infty$-subharmonious}
if for any $z \in \X$ there exists $\delta > 0$ such that for all $r \in (0, \delta]$
\begin{equation}
\label{infinity convex}
2 f(z) \le \sup_{B_r(z)} f+\inf_{B_r(z)} f.
\end{equation}
\end{defi}

As suggested by the naming of such a notion,
the convexity inequality \eqref{infinity convex} amounts to saying that the infinity Laplacian operator $\Delta_\infty$ acted on $f$ in $\L^2$ is nonnegative, which is related to the so-called tug-of-war game interpretation of infinity Laplace equation; see for example \cite{PSSW, AS}.



Inspired by the property \eqref{infinity convex}, we provide another notion of convexity restricted to the ``interior'' of $\X$ in a geodesic manner.

A point $z\in X$ is called a \emph{geodesic interior point} if for any $r > 0$ sufficiently small and for any $x\in B_r(z)\subset \X$, there exists $y\in B_r(z)$ such that $z\in M(x, y)$.
Let $\X^\circ$ denote the set of all geodesic interior points in $\X$.

\begin{defi}[Pointwise convexity]
\label{def point-convex}
A continuous function $f$ in a geodesic space $(\X, d)$ is said to be \emph{pointwise convex}
if for any $z\in \X^\circ$ and any $r > 0$ sufficiently small, there exist $x, y \in B_r(z) \setminus \{ z \}$ such that $z \in M(x, y)$ and \eqref{eq convex} holds.
\end{defi}

Roughly speaking, this weak notion asks for convexity only in one direction at every interior point. 
It is clear that for any $f\in C^2(\R)$,  its $\infty$-subharmoniousness is equivalent to the pointwise convexity. When $\X\neq \X^\circ$, our definition does not require the convexity inequality \eqref{eq convex} to hold with $z\in \X\setminus \X^\circ$.

The following result shows that $\infty$-subharmonious functions are pointwise convex.

 \begin{prop}[Infinity-subharmoniousness implies pointwise convexity]\label{prop har-pt}
Let $(\X, d)$ be a locally compact geodesic space.
If $f\in C(\X)$ is $\infty$-subharmonious, then $f$ is pointwise convex.
\end{prop}

We omit the proof, since it follows from Definitions \ref{def wp-convex} and \ref{def point-convex} in a straightforward manner. 
The reverse implication however does not hold in general, especially when $\X\neq \X^\circ$. For example, letting $\X=\Rp$ with the Euclidean metric and $f(x)=-x$ for $x\in \Rp$, we easily see that $f$ is pointwise convex in $\Rp$ but fails to be $\infty$-subharmonious, since \eqref{infinity convex} does not hold at $z=0$.




For our convenience later, we also introduce a uniform version of Definition \ref{def wp-convex}.

\begin{defi}[Uniform infinity-subharmoniousness]
\label{def u wp-convex}
For any given $\delta>0$, we say a function $f$ on a metric space $\X$ is \emph{uniformly $\infty$-subharmonious} with respect to $\delta$
if for any $z \in \X$ the inequality \eqref{infinity convex} holds for all $r \in (0, \delta]$.
\end{defi}


We conclude this section with an elementary property of geodesic spaces that will be used several times in this paper.



\begin{prop}
\label{thm separation}
Let $(\X, d)$ be a geodesic space.
Then, for every $x, y, z \in \X$ there exists $w \in \X$ such that $d(x, w) \le d(y, z)$, $d(w, z) \le d(x, y)$
and, moreover, the equality of either formula holds.
\end{prop}

\begin{proof}
When $d(x, y) \ge d(y, z)$, take a point $w$ on a geodesic between $x$ and $y$ with $d(x, w) = d(y, z)$.
We then see that
\[
d(w, z) \le d(w, y)+d(y, z) = d(w, y)+d(x, w) = d(x, y).
\]
We may apply a symmetric argument when $d(x, y) \le d(y, z)$.  We pick a point $w$ on a geodesic between $y$ and $z$ with $d(w, z) = d(x, y)$ so that
\[
d(x, w) \le d(x, y)+d(y, w) = d(w, z)+d(y, w) = d(y, z).
\]
\end{proof}

\section{Convexity of metric spaces}
\label{sec convex space}

Before studying the convexity preserving properties for the Hamilton-Jacobi equations,
let us review assumptions on convexity of the metric in the sense of Busemann \cite{Jbook, Pa}.


\begin{defi}[Busemann spaces]
A geodesic space $(\X, d)$ is called a \emph{Busemann space}
if for any $x, y, y' \in \X$ we have
\begin{equation}
\label{eq busemann ineq}
2 d(z, z') \le d(y, y') \quad \text{for all $z\in M(x, y)$ and $z'\in M(x, y')$.}
\end{equation}
\end{defi}

\begin{defi}[Busemann NPC spaces]\label{def npc space}
A geodesic space $(\X, d)$ is called a \emph{nonpositively curved space in the sense of Busemann} (\emph{Busemann NPC space} for short)
if for every $p \in \X$ there exists $\delta > 0$ such that for any $x, y, y' \in B_{\delta}(p)$ the inequality \eqref{eq busemann ineq} holds.
\end{defi}

It is known (cf. \cite[Corollary 2.3.1]{Jbook}) that any Alexandrov NPC space is a Busemann space; see for example \cite{St} for an introduction of Alexandrov NPC spaces and their properties. 

For our convenience of applications, we consider a stronger version of Busemann NPC spaces in the sense that we further require the size of $\delta>0$ in Definition \ref{def npc space} be independent of the location of $p$.

\begin{defi}[Uniform Busemann NPC spaces]\label{weak busemann}
We say that a geodesic space $(\X, d)$ is a \emph{uniform Busemann NPC space}
if there exists $\delta > 0$ such that for any $x, x', y \in B_\delta(p)$ with some $p \in \X$ the inequality \eqref{eq busemann ineq} holds.
\end{defi}



\begin{rmk}
Considering the specific case $y' = y$ in \eqref{eq busemann ineq}, we have $d(z, z') = 0$ for any $z, z' \in M(x, y)$,
which shows that $M(x, y)$ is singleton.
In particular, every midpoint set $M(x, y)$ of a Busemann space consists of only one point.
In the same way, we see that geodesics joining any $x$ and $y$ on a uniform Busemann NPC space $\X$ are unique
provided $d(x, y) \le \delta$ with the constant $\delta$ appearing in Definition \ref{weak busemann}.
\end{rmk}




\begin{example}
\label{ex busemann}
(1)
The Euclidean spaces $\R^N$ with $p$-norms ($1 < p < \infty$)
are Busemann spaces.
Similarly, the Lebesgue spaces $L^p(\R^N)$ ($1 < p < \infty$) are Busemann spaces.



(2)
Hyperbolic spaces are Busemann spaces. 


(3)
Also, any tree or, more generally, any Euclidean Bruhat-Tits building is a Busemann space. 
\end{example}

\begin{example}
\label{rmk heisenberg}
The first Heisenberg group $\mathbb{H}$ with the Carnot-Carath\'eodory metric is neither a Busemann space nor a Busemann NPC space, since it is not a uniquely geodesic space even locally; it is well known that there are infinitely many geodesics joining every pair of points $(x, y, z_1), (x, y, z_2)\in \H$ for any $x, y, z_1, z_2\in \R$ with $z_1\neq z_2$.

\end{example}

It is obvious that Busemann spaces are always uniform Busemann NPC spaces
but the reverse is not true in general.
For example, the flat torus $\T = \R/\Z$ is not a Busemann space since there exist two geodesics between the points $0$ and $1/2$, while it is a uniform Busemann NPC space.
Such an example of metric spaces can be easily modified to be unbounded, as described in the following example. 

\begin{example}
\label{ex npc only}
As discussed in \cite[Example 8.3]{St}, surfaces of revolution
\[
\{(x_1, x_2, x_3)\in \R^3 \mid x_1^2+x_2^2=\varphi^2(x_3), x_3\in I\} \quad \text{with an interval $I\subset \R$}
\]
are locally NPC spaces (in both Alexandrov sense and Busemann sense) when $\varphi: I\to \Rp$ is convex, since their sectional curvatures are nonnegative. In what follows we present a detailed verification, using Definitions \ref{def npc space} and \ref{weak busemann}, in the special case of cylinders. 
Consider the cylinder
\[
\X = \S^1\times \R = \{ (x_1, x_2, x_3) \in \R^3 \mid x_1^2+x_2^2 = 1 \}
\]
equipped with the intrinsic metric $d$ given by
\begin{equation}
\label{cylinder metric}
d((\theta_x, x_3), (\theta_y, y_3))
= \min_{n \in \Z}\sqrt{(\theta_y-\theta_x+2\pi n)^2+(y_3-x_3)^2}.
\end{equation}
Here any $x \in \X$ with coordinates $(x_1, x_2, x_3) \in \R^3$ are identified by cylindrical coordinates  $(\theta_x, x_3) \in \R\times \R$;
in other words, we keep the third coordinate $x_3$ and let $\cos(\theta_x) = x_1$, $\sin(\theta_x) = x_2$.
The midpoint set $M((\theta_x, x_3), (\theta_y, y_3))$ is all points represented by
\[
\left({\theta_x+\theta_y+2\pi n \over 2}, {x_3+y_3 \over 2}\right)
\]
with the minimizers $n$ in \eqref{cylinder metric}.
For instance, when $x = (0, 0)$ and $y = (\pi, 0)$, the minimizers are $n = 0, -1$ and so
\[
M((0, 0), (\pi, 0)) = \{ (\pi/2, 0), (3\pi/2, 0) \}.
\]
In particular, $(\S^1\times \R, d)$ is not a Busemann space.

Meanwhile, for any $x = (\theta_x, x_3), y = (\theta_y, y_3), y' = (\theta_{y'}, y'_3) \in \S^1\times \R$ satisfying $d(x, y), d(x, y') \le \pi/2$,
we see that
\[
\min_{n \in \Z}|\theta_y-\theta_x+2\pi n| \le \pi/2,
\quad \min_{n \in \Z}|\theta_{y'}-\theta_x+2\pi n| \le \pi/2
\]
and hence
\[
\theta_x-\pi/2 \le \theta_y+2\pi n \le \theta_x+\pi/2,\
\quad \theta_x-\pi/2 \le \theta_{y'}+2\pi n' \le \theta_x+\pi/2
\]
for some unique $n, n' \in \Z$.
Moreover, $z := ((\theta_x+\theta_y+2\pi n)/2, (x_3+y_3)/2)$ and $z' := ((\theta_x+\theta_{y'}+2\pi n)/2, (x_3+y_3)/2)$ are the unique element of $M(x, y)$ and $M(x, y')$, respectively.
Without loss of generality, we may assume $n, n' = 0$.
Then, since $|\theta_y-\theta_{y'}| \le \pi$, we see that
\[
d(y, y') = \sqrt{(\theta_{y'}-\theta_y)^2+(y'_3-y'_3)^2},
\quad d(z, z') = \sqrt{\left({\theta_{y'}-\theta_y \over 2}\right)^2+\left({y'_3-y'_3 \over 2}\right)^2},
\]
which shows $2 d(z, z') = d(y, y')$ for all $z \in M(x, y)$ and $z' \in M(x, y')$.
Therefore, we see that $(\S^1\times \R, d)$ is a uniform Busemann NPC space.
\end{example}

The condition \eqref{eq busemann ineq} is equivalent to the one involving four points which has been introduced as \eqref{eq busemann ineq 4}.

\begin{prop}
\label{thm busemann 3and4}
Let $(\X, d)$ be a geodesic space and let $B$ be a subset of $X$.
Then, \eqref{eq busemann ineq} holds for any $x, y, y'\in B$
if and only if \eqref{eq busemann ineq 4} holds for any $x, x', y, y' \in B$.
\end{prop}

\begin{proof}
The implication ``$\eqref{eq busemann ineq 4}\Rightarrow\eqref{eq busemann ineq}$'' is trivial.

We prove the opposite implication.
Fix arbitrary $z \in M(x, y)$, $z' \in M(x', y')$ and $w \in M(x, y')$.
Now, applying the inequality \eqref{eq busemann ineq} to the triples $(x, y, y')$ and $(x, x', y')$
we can get
\[
2 d(z, w) \le d(y, y'),
\quad 2 d(w, z') \le d(x, x').
\]
Therefore, we obtain
\[
2 d(z, z') \le 2(d(z, w)+d(w, z')) \le d(x, x')+d(y, y'),
\]
which shows \eqref{eq busemann ineq 4}.

\end{proof}

\section{Metric viscosity solutions and Hopf-Lax formula}\label{sec viscosity}

This section is devoted to a review of the notion of metric viscosity solutions and Hopf-Lax formula to the Hamilton-Jacobi equation \eqref{hj1}.
For more details we refer the reader to \cite{GaS2, GaS}.

Let $(\X, d)$ be a complete geodesic space and $x_0 \in X$ be a fixed point.
Since a part of our study relies on the Hopf-Lax formula, for the moment we impose the following assumptions on $H$: 
\begin{enumerate}
\item[(H1)]
$H \in C(\Rp)$ and $H(0) = 0$.
\item[(H2)]
$H$ is non-decreasing convex in $\Rp$.
\item[(H3)]
$H$ is coercive; that is 
\begin{equation}
\label{coercive}
{H(p) \over p}\to \infty \quad \text{as $p\to \infty$.}
\end{equation}
\end{enumerate}
Assume $u_0$ is Lipschitz continuous; in other words, $|u_0(x)-u_0(y)| \le K d(x, y)$ for all $x, y \in \X$ with some $K < \infty$.
In particular, it follows that $|u_0(x)| \le C (d(x_0, x)+1)$ for some $C < \infty$.


We denote by $\mathcal{C}$ a set of locally Lipschitz continuous functions $u$ on $\X\times(0, \infty)$ such that $\partial_{t}u$ is continuous,
and introduce the sets of smooth functions on metric spaces as below:
\[
\begin{aligned}
\overline{\mathcal{C}} &:= \{ u \in \mathcal{C} \mid \text{$|\nabla^+ u| = |\nabla u|$ and $|\nabla u|$ is continuous} \}, \\
\underline{\mathcal{C}} &:= \{ u \in \mathcal{C} \mid \text{$|\nabla^- u| = |\nabla u|$ and $|\nabla u|$ is continuous} \}, \\
\end{aligned}
\]
where
\[
|\nabla^\pm u|(x, t) := \limsup_{y \to x} \frac{[u(y, t)-u(x, t)]_\pm}{d(x, y)}
\]
with $[\cdot]_+:=\max\{\cdot,0\}$ and $[\cdot]_-:=-\min\{\cdot,0\}$.


\begin{defi}[Metric viscosity solutions]
An upper semicontinuous function $u$ on $\X\times(0, \infty)$ is said to be a \emph{metric viscosity subsolution} of \eqref{hj1}
when for every $\psi = \psi_1+\psi_2$ with $\psi_1 \in \underline{\mathcal{C}}$ and $\psi_2 \in \mathcal{C}$,
if $u-\psi$ attains a local maximum at a point $(x, t) \in \X\times(0, \infty)$,
then
\[
\partial_t\psi(x, t)+H_{|\nabla \psi_2|^*(x, t)}(|\nabla \psi_1|(x, t)) \le 0,
\]
where $H_a(p) = \inf_{|p'-p| \le a}H(p')$ for $a \ge 0$
and $|\nabla \psi_2|^*(x, t) = \limsup_{(x', t') \to (x, t)}|\nabla \psi_2|(x', t')$.

Similarly,
a lower semicontinuous function $u$ on $\X\times(0, \infty)$ is to be said a \emph{metric viscosity supersolution} of \eqref{hj1}
when for every $\psi = \psi_1+\psi_2$ with $\psi_1 \in \overline{\mathcal{C}}$ and $\psi_2 \in \mathcal{C}$,
if $u-\psi$ attains a local minimum at a point $(x, t) \in \X\times(0, \infty)$,
then
\[
\partial_t\psi(x, t)+H^{|\nabla \psi_2|^*(x, t)}(|\nabla \psi_1|(x, t)) \ge 0,
\]
where $H^a(p) = \sup_{|p'-p| \le a}H(p')$ for $a \ge 0$.

We say that $u$ is a \emph{metric viscosity solution} of \eqref{hj1}
if $u$ is both a metric viscosity sub- and supersolution of \eqref{hj1}.
\end{defi}

\begin{thm}[{\cite[Section 7]{GaS2}}]
\label{thm solhj}
Suppose that $(\X, d)$ is a complete geodesic space. Assume that $H$ satisfies (H1) and that $u_0$ is Lipschitz continuous.
Then, there exists a unique metric viscosity solution $u$ of (HJ) with the $C(\X\times[0, \infty))$ regularity and the growth condition
\begin{equation}
\label{eq growth}
|u(x, t)| \le C(d(x_0, x)+t+1)
\end{equation}
for all $(x, t)\in \X\times[0, \infty)$ and some $C < \infty$.
Moreover, if (H2) and (H3) hold,
then the solution $u$ is given by the Hopf-Lax formula \eqref{eq hl}
with the function $L \colon \Rp \to \R$ defined by
\begin{equation}\label{lagrangian}
L(v) = \sup_{p \in \Rp}(p v-H(p)).
\end{equation}
\end{thm}

\begin{rmk}
The function $L: \Rp\to \R$ is increasing continuous convex, $L(0) = 0$ and satisfies the coercivity
\begin{equation}
\label{eq suplinear l}
{L(v)\over v} \to \infty\quad \text{as $v \to \infty$.}
\end{equation}
\end{rmk}

Let us give a few properties of Hopf-Lax formula \eqref{eq hl} that are needed later.

\begin{prop}[Finite speed of propagation]
\label{thm finiteprop}
Suppose that $(\X, d)$ is a complete geodesic space. 
Let $H$ satisfy (H1)--(H3) and $L$ be given as in \eqref{lagrangian}.
Assume that $u_0$ is a $K$-Lipschitz continuous function.
Then, the function $u$ defined by \eqref{eq hl} satisfies
\begin{equation}
\label{eq hl2}
u(x, t) = \inf_{a \in B_{V t}}\left\{ u_0(a)+t L\left({d(a, x) \over t}\right) \right\} \quad \text{for all $(x, t) \in \X\times(0, \infty)$}
\end{equation}
with some $V \in \Rp$ depending only on the Lipschitz constant $K$ of $u_0$ and the Lagrangian $L$.
\end{prop}

\begin{proof}
Since $u(x, t) \le u_0(x)$,
for $a \in \X$ satisfying
\begin{equation}
\label{ineq hl}
u(x, t) \ge u_0(a)+t L\left({d(a, x) \over t}\right)
\end{equation}
we observe that
\[
t L\left({d(a, x) \over t}\right) \le u_0(x)-u_0(a) \le K d(a, x).
\]
Hence, setting $v = d(x, a)/t$ we have $L(v)/v \le K < \infty$.
Now, thanks to the condition \eqref{eq suplinear l} we obtain $V \in \Rp$ determined only by $L$ and $K$ such that $v \le V$, i.e.\ $a \in B_{V t}(x)$.
Note that we have proved that \eqref{ineq hl} implies $a \in B_{V t}(x)$
and therefore we obtain \eqref{eq hl2}.
\end{proof}

In view of the Hopf-Lax formula we can obtain the Lipschitz preserving property below by adapting the standard argument to the current general metric space.

\begin{lem}[Lipschitz preserving property]
\label{lip-preserving}\label{thm lip pre}
Suppose that $(\X, d)$ is a complete geodesic space. 
Assume that $H$ satisfies (H1)--(H3) and that $u_0$ is Lipschitz continuous in $\X$.
Let $u$ be the unique viscosity solution of (HJ).
If $u_0$ is $K$-Lipschitz continuous,
then all $u(\cdot, t)$ are $K$-Lipschitz continuous,
i.e.\
\[
|u(x, t)-u(y, t)| \le K d(x, y)
\]
for all $x, y\in X$ and all $t \ge 0$.
\end{lem}

\begin{proof}
For fixed $x, y \in \X$ and $\vep > 0$ take $b \in \X$ such that
\begin{equation}
\label{eq lip pre 1}
u(y, t) \ge u_0(b)+t L\left({d(b, y) \over t}\right)-\vep.
\end{equation}
Since $(\X, d)$ is a geodesic space,
we see by Proposition \ref{thm separation} that there is $a \in \X$ satisfying
\begin{equation}
\label{eq lip pre 3}
d(a, b) \le d(x, y), \quad d(a, x) \le d(b, y).
\end{equation}
Note that
\begin{equation}
\label{eq lip pre 2}
u(x, t) \le u_0(a)+t L\left({d(a, x) \over t}\right).
\end{equation}
Subtracting \eqref{eq lip pre 1} from \eqref{eq lip pre 2},
we have
\[
u(x, t)-u(y, t) \le u_0(a)+t L\left({d(a, x) \over t}\right)-u_0(b)-t L\left({d(b, y) \over t}\right)+\vep.
\]
Since $u_0$ is $K$-Lipschitz continuous and $L$ is non-decreasing,
we see by \eqref{eq lip pre 3} that
\[
u(x, t)-u(y, t) \le K d(x, y)+\vep,
\]
which completes the proof
because $x, y, \vep$ are arbitrary.
\end{proof}



 

The assumption (H3) on superlinear growth of $H$ guarantees that the function $L$ given by \eqref{lagrangian} is real-valued continuous.
On the other hand, it is known that a similar formula based on optimal control holds for the case when $H$ has linear growth at infinity
(see for example \cite{BC} on the control-based interpretation of Hamilton-Jacobi equations in the Euclidean spaces).
We here present such a representation formula on metric spaces in the simplest case.

\begin{thm}[Hopf-Lax formula for the time-dependent Eikonal equation]
\label{thm hopf-lax1}
Suppose that $(\X, d)$ is a complete geodesic space. Let $u_0$ be a Lipschitz continuous function on $\X$.
Then, the function $u$ given by
\begin{equation}
\label{eq hl1}
u(x, t) = \inf_{a \in B_t(x)} u_0(a)
\end{equation}
is the unique metric viscosity solution of
\begin{equation}
\label{c}
\partial_t u+|\nabla u|=0 \quad \text{in $\X\times (0, \infty)$}
\end{equation}
with initial condition \eqref{hj2}.
\end{thm}

We leave the detailed proof of this result in the Appendix.

\section{Convexity preserving properties}
\label{sec main}

In this section we give our main results on convexity preserving properties of the equations (HJ) and (HJ2).

\subsection{A metric/Lagrangian approach to geodesic convexity preserving}

\begin{thm}[Preservation of geodesic convexity in uniform Busemann NPC spaces]
\label{thm convexity1}
Let $(\X, d)$ be a complete uniform Busemann NPC space.
Assume that $H$ satisfies (H1)--(H3) and that $u_0$ is Lipschitz continuous in $\X$.
Let $u$ be the unique viscosity solution of (HJ).
If $u_0$ is weakly geodesically convex,
then $u(\cdot, t)$ is also weakly geodesically convex for all $t \ge 0$.
\end{thm}

\begin{proof}
Since $(\X, d)$ is a uniform Busemann NPC space, let $\delta > 0$ be the constant that appears in Definition \ref{weak busemann}.
Let $t > 0$ and fix $x, y \in \X$ with $d(x, y) \le \delta/2$.
In view of the Hopf-Lax formula \eqref{eq hl2}, for every $\vep > 0$ there exist $a \in B_{V t}(x)$ and $b \in B_{V t}(y)$ such that
\begin{equation}
\label{eq:hopflax1}
u(x, t) \ge u_0(a)+t L\left({d(a, x) \over t}\right)-\vep,
\quad u(y, t) \ge u_0(b)+t L\left({d(b, y) \over t}\right)-\vep,
\end{equation}
where $V$ is the constant given in Proposition \ref{thm finiteprop}.
Let us restrict ourselves to the case when $t \le t_0 := \delta/2 V > 0$ to guarantee $x, y, a, b \in B_{\delta}(x)$.

Since $u_0$ is weakly geodesically convex,
there exists $c \in M(a, b)$ satisfying
\begin{equation}
\label{eq:convex1}
2 u_0(c) \le u_0(a)+u_0(b)+\vep.
\end{equation}
On the other hand, since $\X$ is a uniform Busemann NPC space, for any $z \in M(x, y)$, we have
\[
2 d(c, z) \le d(a, x)+d(b, y),
\]
which, due to the monotonicity and convexity of $L$, yields that
\begin{equation}
\label{eq:convex2}
2 L\left({d(c, z) \over t}\right)
\le 2 L\left({1 \over 2}{d(a, x) \over t}+{1 \over 2}{d(b, y) \over t}\right)
\le L\left({d(a, x) \over t}\right)+L\left({d(b, y) \over t}\right).
\end{equation}
Then combining \eqref{eq:convex1}, \eqref{eq:convex2} and \eqref{eq:hopflax1},
we obtain
\[
2 u(z, t)
\le 2 u_0(c)+2 t L\left({d(z, c) \over t}\right)
\le u(x, t)+u(y, t)+3\vep.
\]
Since $\vep > 0$ is arbitrary,
$u(\cdot, t)$ is locally weakly geodesically convex 
and therefore it is also weakly geodesically convex for all $0 < t \le t_0$, in view of Lemma \ref{local convexity}.

Now we may repeat our argument above, treating $t_0$ as the initial moment.
An alternative viewpoint is to use the dynamic programming principle 
\[
u(x, t) = \inf_{a \in \X}\left\{u(a, s)+(t-s)L\left({d(a, x) \over t-s}\right)\right\}
\]
for all $x \in \X$ and $0 \le s < t$.
Noticing that the Lipschitz constant for $u(\cdot, t)$ does not depend on $t$ by Proposition \ref{thm lip pre},
we obtain the weak geodesic convexity of $u(\cdot, t)$ for an arbitrary $t > 0$.
The case $t = 0$ is trivial because $u(\cdot, 0) = u_0$.
\end{proof}

\begin{rmk}
Since the key to our proof is the relation \eqref{eq:convex2}, it is possible to weaken the assumptions; one may simply assume that for a given Lagrangian $L$ and $x, x', y, y', z, z'\in \X$  with $z\in M(x, y)$ and $z'\in M(x', y')$, we have
\[
2 L\left({d(z, z') \over t}\right)
\leq L\left({d(x, x') \over t}\right)+L\left({d(y, y') \over t}\right)
\]
for any $t>0$ small.  Such a condition incorporates the Hamiltonian $H$ into the structure of space $\X$ but it is not clear to us how general $\X$ could be under this assumption. 
\end{rmk}





\begin{rmk}

It is easily seen that Theorem \ref{thm convexity1} holds
also when $(\X, d)$ satisfies for every $x, y, x', y'\in X$ and 
\begin{equation}
\label{eq weak busemann ineq 4}
\text{for any $z\in M(x, y)$ there exists $z'\in M(x', y')$ such that $2 d(z, z') \le d(x, x')+d(y, y')$}
\end{equation}
instead of \eqref{eq busemann ineq 4}.

We however do not know whether it is possible to find a concrete example of geodesic convexity preserving on such a more general space that is not a uniform Busemann NPC space. 
\end{rmk}

\begin{rmk}\label{rmk linear growth}
The convexity preserving property also holds when $H$ is not coercive in the sense of \eqref{coercive} (superlinear growth at infinity) but only with linear growth instead.
For instance, let us consider the case when $H(p) = p$.

Then, in view of Theorem \ref{thm hopf-lax1}, we have
\[
u(x, t)+u(y, t) \ge u_0(a)+u_0(b)-\vep
\]
for certain $a \in B_t(x)$ and $b \in B_t(y)$.
By the assumptions, we get $z \in M(x, y)$ and $c \in M(a, b)$ with $2 d(c, z) \le d(a, x)+d(b, y) \le 2 t \le \delta$ satisfying
\[
2 u(z, t) \le u_0(c) \le u(x, t)+u(y, t)+3\vep
\]
for any $t \le \delta/2$.
An iteration of this argument for finite steps concludes the proof for an arbitrary $t \ge 0$.

\begin{example}\label{ex npc only2}
Based on Example \ref{ex npc only}, we provide a concrete but very simple example, where (HJ) preserves geodesic convexity of the solution in a uniform Busemann NPC space. Recall that the metric space $\X=\S^1\times \R$ is equipped with the metric given as in \eqref{cylinder metric}. Let $(\theta_x, x_3)$ denote the cylindrical coordinates of any $x\in\X$. 

We now take 
\[
u_0(x)=x_3
\]
for any $x=(\theta_x, x_3)\in \X$. It is quite clear that $u_0$ is weakly (and also strongly) geodesically convex in $\X$. 

The unique metric viscosity solution of (HJ) with $H(p) = p$ in this case is clearly 
\[
u(x, t)=x_3-t
\]
for $x\in \X$ and $t\geq 0$, which is certainly a geodesically convex function in $x$ for all $t\geq 0$. 

\end{example}


\end{rmk}

\subsection{A PDE approach to geodesic convexity preserving}

Motivated by \cite{K2, Kor, GGIS, Ju} etc, we here provide an alternative proof for preservation of geodesic convexity from the PDE perspectives in the general framework of metric viscosity solutions established in \cite{GaS2, GaS}. In order to make our argument work, we assume that $\X$ is a Busemann space, which means that $\X$ is uniquely geodesic. 

\begin{thm}[Preservation of geodesic convexity in Busemann spaces]
\label{thm convexity2}
Assume that $(\X, d)$ is a complete Busemann space.
Assume that $H \colon \Rp \to \R$ is continuous and non-decreasing.
Let $u$ be the unique metric viscosity solution of (HJ) with $u_0$ Lipschitz continuous in $\X$.
If $u_0$ is geodesically convex,
then $u(\cdot, t)$ is also geodesically convex for all $t \ge 0$.
\end{thm}

In contrast to Theorem \ref{thm convexity1}, we do not assume the convexity of $H$ in Theorem \ref{thm convexity2}. However, the assumptions on the structure of the metric space become stronger. It is not clear to us whether one can apply our PDE method to more general spaces such as uniform Busemann NPC spaces without strengthening the assumptions on $H$. 


The next lemma plays an important role in our proof.

\begin{lem}
\label{thm diff mid}
Assume that $(\X, d)$ is a complete Busemann space
and let $\varphi \in C^1([0, \infty))$.
Then,
$$
|\nabla_x \varphi(d(z, m(x, y)))| \le {1 \over 2}|\varphi'(d(z, m(x, y)))|
$$
for all $x, y, z \in \X$.
\end{lem}

\begin{proof}
Note that $(x, y, z)\mapsto d(z, m(x, y))$ is Lipschitz continuous in $\X^3$.
Indeed, since $(\X, d)$ is a Busemann space, we have
\[
\begin{aligned}
|d(z', m(x', y'))-d(z, m(x, y))|
&\le d(z', z)+d(m(x', y'), m(x, y)) \\
&\le d(z', z)+{1 \over 2}d(x', x)+{1 \over 2}d(y', y).
\end{aligned}
\]
We next observe
\[
\begin{aligned}
&\varphi(d(z, m(x', y)))-\varphi(d(z, m(x, y))) \\
&\ = \varphi'(d(z, m(x, y)))(d(z, m(x', y))-d(z, m(x, y)))+o(|d(z, m(x', y))-d(z, m(x, y))|) \\
&\ \le {1 \over 2}\varphi'(d(z, m(x, y)))d(x', x)+o(d(x', x))
\end{aligned}
\]
for $x'$ near to $x$.
Therefore,
\[
\begin{aligned}
|\nabla_x \varphi(d(z, m(x, y)))|
&= \limsup_{x' \to x}{|\varphi(d(z, m(x', y)))-\varphi(d(z, m(x, y)))| \over d(x', x)} \\
&\le {1 \over 2}|\varphi'(d(z, m(x, y)))|.
\end{aligned}
\]
The proof is now completed.
\end{proof}

Before starting the proof of Theorem \ref{thm convexity2},
we establish the following estimate.

\begin{prop}
\label{thm conv grow}
Under the same assumptions as in Theorem \ref{thm convexity2},
let $u$ be the unique metric viscosity solution of (HJ).
Then there exists $C < \infty$ such that
\begin{equation}
\label{eq conv grow}
2 u(z, t)-u(x, r)-u(y, s) \le C\left(d(z, m(x, y))+t+s+r\right)
\end{equation}
for all $x, y, z \in \X$ and all $t, s, r \ge 0$.
In particular,
\begin{equation}\label{eq conv grow2}
2 u(z, t)-u(x, t)-u(y, t) \le C(d(z, m(x, y))+3t)
\end{equation}
for all $x, y, z\in \X$ and $t \ge 0$.
\end{prop}

In order to show Proposition \ref{thm conv grow}, we prepare the following result.

\begin{prop}
\label{thm sol mid} Suppose that $(\X, d)$ is a complete Busemann space. 
Set $\psi_{k, C, \vep}(x, y, z, t) = k\langle d(z, m(x, y)) \rangle_\vep+C t$ for $k, C \in \R$ and $\vep > 0$,
where $\langle r \rangle_\vep := \sqrt{r^2+\vep^2}$ for every $r\in \R$.
Then for any $k \in \R$ there exists $C = C_k \in \R$ independent of $\vep$ satisfying the following:
\begin{enumerate}
\item
if $k > 0$, then for any $x, y\in \X$, the function $\psi_1(z, t) = \psi_{k, C, \vep}(x, y, z, t)$ is a metric viscosity supersolution of (HJ) for all $C \ge C_k$; and
\item
if $k < 0$, then for any $y, z\in \X$, the function $\psi_2(x, t) = \psi_{k, C, \vep}(x, y, z, t)$ is a metric viscosity subsolution of (HJ) for all $C \le C_k$.
\end{enumerate}
\end{prop}

\begin{proof}
If $k > 0$, we have $\psi_1 \in \Colm$ and 
\[
\partial_t\psi_1(z, t)+H\left(|\nabla_z\psi_1|(z, t)\right)
= C+H\left(k{d(z, m(x, y)) \over \langle d(z, m(x, y)) \rangle_\vep}\right)
\ge C+\inf_{0 \le p \le k}H(p).
\]
Therefore, setting 
\[
C_k = -\inf_{0 \le p \le k}H(p),
\] 
we see that $\psi_1$ is a metric viscosity solution of (HJ) by \cite[Lemma 2.8]{GaS}.

If $k < 0$, by Lemma \ref{thm diff mid}, we obtain that
\[
|\nabla_x\psi_2|(x, t) \le {|k| \over 2}{d(z, m(x, y)) \over \langle d(z, m(x, y)) \rangle_\vep} \le {|k| \over 2}.
\]
It follows that
\[
\partial_t\psi_2(x, t)+H^{|\nabla_x\psi_2|^*(x, t)}\left(0\right)
\le C+\sup_{0 \le p \le |k|/2}H(p).
\]
Therefore $\psi_2$ is a metric viscosity solution of (HJ) if we take
\[
C_k = -\sup_{0 \le p \le -k/2}H(p),
\]
thanks to \cite[Lemma 2.8]{GaS} again.
\end{proof}

\begin{proof}[Proof of Proposition \ref{thm conv grow}]
Since $u_0$ is geodesically convex and Lipschitz continuous,
we have
\[
2 u_0(z)-u_0(x)-u_0(y)
\le 2 u_0(z)-2 u_0(m(x, y))
\le 2K d(z, m(x, y)),
\]
where $K < \infty$ is the Lipschitz constant of $u_0$.
Proposition \ref{thm sol mid} provides a constant $C \in \Rp$ such that $(z, t) \mapsto K\langle d(z, m(x, y)) \rangle_\vep+C t/2$ is a metric viscosity supersolution of (HJ),
and $(x, t) \mapsto -2 K\langle d(z, m(x, y)) \rangle_\vep-C t$ and $(y, t) \mapsto -2 K\langle d(z, m(x, y)) \rangle_\vep-C t$ are metric viscosity subsolutions of (HJ) for all $\vep > 0$.

Now, since
\[
2 u(z, 0)-u(x, 0)-u(y, 0)
\le 2K\langle d(z, m(x, y)) \rangle_\vep+C\cdot 0
\]
and $u(z, t)$ is a metric viscosity subsolution of (HJ),
the comparison principle \cite[Theorem 7.5]{GaS2} implies that 
\[
2 u(z, t)-u(x, 0)-u(y, 0)
\le 2K\langle d(z, m(x, y)) \rangle_\vep+C t.
\]
Since $u(x, r)$ and $u(y, s)$ are metric viscosity supersolutions of (HJ),
following a similar argument, we obtain
\[
2 u(z, t)-u(x, r)-u(y, s)
\le 2K\langle d(z, m(x, y)) \rangle_\vep+C(r+s+t)
\]
for all $x, y, z \in \X$, $r, s, t \ge 0$.
Therefore, we get the inequality \eqref{eq conv grow} by sending $\vep \to 0$.
The inequality \eqref{eq conv grow2} follows immediately if one takes $r, s=t$ in \eqref{eq conv grow}. 
\end{proof}

\begin{proof}[Proof of Theorem \ref{thm convexity2}]
We assume by contradiction that 
\[
2u(m(\hat{x}, \hat{y}), \hat{t})-u(\hat{x}, \hat{t})-u(\hat{y}, \hat{t}) =: 2\eta > 0
\]
at some $\hat{x}, \hat{y}\in \X$ and $\hat{t} > 0$.
Setting $T = \hat{t}+1$, for $k, \vep, \sigma > 0$, we consider the continuous function
\[
\Phi(x, y, z, t, r, s) = F(x, y, z, t, r, s)-{1 \over \vep}\varphi(t, r, s)-\sigma h(t, r, s) \quad \text{for $x, y, z \in \X$, $t, r, s \in [0, T)$,}
\]
where 
\[
F(x, y, z, t, r, s) = 2 u(z, t)-u(x, r)-u(y, s)-k d(z, m(x, y))^2,
\]
\[
\varphi(t, r, s) = (t-r)^2+(t-s)^2+(r-s)^2 \ge 0,
\]
\[
h(t, r, s) = {1 \over T-t}+{1 \over T-r}+{1 \over T-s} \ge 0.
\]
Note by Proposition \ref{thm conv grow} that
\[
\begin{aligned}
F(x, y, z, t, r, s)
&\le C d(z, m(x, y))-k d(z, m(x, y))^2+C (t+r+s) \\
&\le {C^2\over 4 k}+3C T
< \infty
\end{aligned}
\]
with the constant $C$ given in Proposition \ref{thm conv grow}.
We hence see that $\Phi$ is bounded from above.
Also note that
\[
\Phi(\hat{x}, \hat{y}, \hat{z}, \hat{t}, \hat{t}, \hat{t})
\ge 2\eta-{3\sigma \over T-\hat{t}},
\]
which implies $\sup\Phi \ge \eta$ by taking $\sigma$ small enough.

In view of Ekeland's variational principle, there is a point $(x_\vep, y_\vep, z_\vep, t_\vep, r_\vep, s_\vep) \in \X^3\times[0, T)^3$ such that
\[
\Phi(x_\vep, y_\vep, z_\vep, t_\vep, r_\vep, s_\vep) \ge \sup\Phi-\vep^2
\]
and
\[
\Phi(x, y, z, t, r, s)-\vep R(x, y, z)
\]
attains a maximum at $(x_\vep, y_\vep, z_\vep, t_\vep, r_\vep, s_\vep)$ with
$
R(x, y, z) = d(x_\vep, x)+d(y_\vep, y)+d(z_\vep, z)
$.
Recalling $\sup\Phi \ge \eta$, we have
\[
\eta-\vep^2 \le C d(z_\vep, m(x_\vep, y_\vep))-k d(z_\vep, m(x_\vep, y_\vep))^2+C(t_\vep+r_\vep+s_\vep)-{1 \over \vep}\varphi(t_\vep, r_\vep, s_\vep),
\]
which implies that
\begin{equation}
\label{max claim3}
d(z_\vep, m(x_\vep, y_\vep)) \le {C+\sqrt{C^2+12k C T} \over 2k},
\end{equation}
\begin{equation}
\label{max claim5}
t_\vep+r_\vep+s_\vep \ge {\eta \over 2C},
\end{equation}
\begin{equation}
\label{max claim4}
\varphi(t_\vep, r_\vep, s_\vep)
= (t_\vep-r_\vep)^2+(t_\vep-s_\vep)^2+(r_\vep-s_\vep)^2 \le 3CT\vep
\end{equation}
for sufficiently small $\vep \le \sqrt{\eta/2}$ and some large $k = C^2/\eta$.
Now, the third inequality \eqref{max claim4} yields the existence of a subsequence of $t_\vep, r_\vep, s_\vep$, still indexed by $\vep$, such that 
$t_\vep, r_\vep, s_\vep \to t_0$ for some $t_0 \ge 0$.
We then see by the second inequality \eqref{max claim5} that $t_0 > 0$.
Hence, $t_\vep, r_\vep, s_\vep$ must all stay away from $0$ when $\vep$ is small enough.

Since $u$ is a metric viscosity solution,
we apply Lemma \ref{thm diff mid} to obtain 
\begin{equation}\label{viscosity ineq1}
{1\over 2\vep}\partial_t\varphi(t_\vep, r_\vep, s_\vep)+{\sigma \over 2}\partial_t h(t_\vep, r_\vep, s_\vep)+H_{\vep}(p_\vep) \le 0,
\end{equation}
\begin{equation}\label{viscosity ineq2}
-{1\over \vep}\partial_r\varphi(t_\vep, r_\vep, s_\vep)-\sigma\partial_r h(t_\vep, r_\vep, s_\vep)+H^{p_\vep+\vep}(0) \ge 0,
\end{equation}
\begin{equation}\label{viscosity ineq3}
-{1\over \vep}\partial_s\varphi(t_\vep, r_\vep, s_\vep)-\sigma\partial_s h(t_\vep, r_\vep, s_\vep)+H^{p_\vep+\vep}(0) \ge 0,
\end{equation}
\smallskip
where $p_\vep := k d(z_\vep, m(x_\vep, y_\vep))$.
By direct calculations, 
\[
\partial_t\varphi(t_\vep, r_\vep, s_\vep)+\partial_r\varphi(t_\vep, r_\vep, s_\vep)+\partial_s\varphi(t_\vep, r_\vep, s_\vep) = 0,
\]
\[
\partial_t h(t_\vep, r_\vep, s_\vep)+\partial_r h(t_\vep, r_\vep, s_\vep)+\partial_s h(t_\vep, r_\vep, s_\vep)
\ge {3 \over T^2}.
\]
Observe that
\begin{equation}\label{observation hamiltonian}
H^{p_\vep+\vep}(0) = \sup_{B_{p_\vep+\vep}(0)}H = H(p_\vep+\vep),
\quad H_\vep(p_\vep) = \inf_{B_{\vep}(p_\vep)}H = H(p_\vep-\vep)
\end{equation}
by the monotonicity assumption of $H$.

Combining the viscosity inequalities above (subtracting twice \eqref{viscosity ineq1} from the sum of \eqref{viscosity ineq2} and \eqref{viscosity ineq3}), we have, due to \eqref{observation hamiltonian}, 
\begin{equation}
\label{eq convpde2}
-{3\sigma \over T^2}+2(H(p_\vep+\vep)-H(p_\vep-\vep)) \ge 0.
\end{equation}
Recalling \eqref{max claim3}, we have the uniform boundedness of $p_\vep$ in $\vep$.
We therefore obtain a contradiction by passing to the limit in \eqref{eq convpde2} as $\vep \to 0$.
\end{proof}

\begin{rmk}\label{rmk monotonicity}
The monotonicity of $H$ is necessary for the result in Theorem \ref{thm convexity2}. 
For example, when $\X$ is the half line $\Rp$ with the usual Euclidean metric, the function 
\[
u(x, t) = \min\{ t-x, 0 \} \quad \text{ for $(x, t)\in \X\times [0, \infty)$}
\] 
is the unique metric viscosity solution of $\partial_t u-|\nabla u|=0$.
It is easily seen that although the initial value $u(x, 0) = -x$ is geodesically convex in $\X$, the solution itself is not at any $t > 0$. 

The main reason why we need the monotonicity in the proof is that the penalty function $\psi(x) = d(z, m(x, y))^2$ may not be of the class $\Colm$.
Indeed, when $\X=\Rp$, since for $x\in \X$,
\[
d(0, m(x, 1))^2-d(0, m(0, 1))^2 = (1+x)^2/4-1/4 \ge 0
\]
we see that $|\nabla^-\psi|(0) = 0$ while $|\nabla\psi|(0) = 1/2$ when $z = 0, y = 1$.

Moreover, it is not difficult to verify that the metric viscosity solution $u$ above in this special case solves the corresponding state constraint problem 
\begin{numcases}{}
\partial_t u(x, t)-|\nabla u(x, t)| =0 &\text{for $x>0$ and $t>0$,} \nonumber \\
\partial_t u(0, t)-|\nabla u(0, t)| \leq 0 &\text{for  $t>0$,} \nonumber\\
u(\cdot , 0)=u_0 &\text{in $\Rp$}  \nonumber
\end{numcases}
in the viscosity sense.  Our example therefore indicates that the convexity preserving property may not hold for time-dependent state constraint problems even in the Euclidean space, although, as shown in \cite{ALL}, solutions with state constraints are convex in the stationary case.
\end{rmk}

\subsection{Preservation of infinity-subharmoniousness}

In what follows, we switch the sign of the Hamiltonian; in other words, we consider (HJ2).

\begin{thm}[Preservation of uniform $\infty$-subharmoniousness]
\label{thm convexity3}
Let $(\X, d)$ be a complete geodesic space.
Assume that $H$ satisfies (H1)--(H3).
Let $u_0$ be Lipschitz continuous in $\X$ and $u$ be the unique metric viscosity solution of (HJ2).
If $u_0$ is uniformly $\infty$-subharmonious with respect to $\delta>0$,
then so is $u(\cdot, t)$ for all $t \ge 0$.
\end{thm}

\begin{proof}

Our proof is again based on the Hopf-Lax formula for the solution $u$ of (HJ2), which, in the present case, is written as
\begin{equation}\label{hopf-lax2}
u(x, t)=\sup_{a \in X} \left\{u_0(a)-t L\left({d(a, x)\over t}\right)\right\}
\end{equation}
for every $a\in \X$ and $t > 0$, where $L$ is given as in \eqref{lagrangian}.

Fix $r > 0$ and $t > 0$.
For any $z \in \X$ and $\vep>0$, there exists $c_\vep \in \X$ such that
\begin{equation}
\label{eq subhpres1}
u(z, t) \leq u_0(c_\vep)-t L\left({d(c_\vep, z) \over t}\right)+\vep.
\end{equation}
Since $(\X, d)$ is a geodesic space, we see by Proposition \ref{thm separation} that for any $a \in B_r(c_\vep)$ there is $x \in \X$ satisfying $d(z, x) \le d(c_\vep, a)$ and $d(x, a) \le d(z, c_\vep)$.
In particular, $x \in B_r(z)$ and by the monotonicity of $L$ we see that
\[
u(x, t) \ge u_0(a)-t L\left({d(a, x) \over t}\right) \ge u_0(a)-t L\left({d(c_\vep, z) \over t}\right).
\]
Therefore,
\begin{equation}
\label{eq subhpres2}
\sup_{B_r(z)}u(\cdot, t) \ge \sup_{B_r(c_\vep)}u_0-t L\left({d(c_\vep, z) \over t}\right).
\end{equation}
Similarly, for any $y \in B_r(z)$ there is $b \in \X$ satisfying $d(c_\vep, b) \le d(z, y)$ and $d(b, y) \le d(c_\vep, z)$.
Hence, $b \in B_r(c_\vep)$ and
\[
u(y, t) \ge u_0(b)-t L\left({d(b, y) \over t}\right) \ge u_0(b)-tL\left({d(c_\vep, z) \over t}\right).
\]
Therefore,
\begin{equation}
\label{eq subhpres3}
\inf_{B_r(z)}u(\cdot, t) \ge \inf_{B_r(c_\vep)}u_0-t L\left({d(c_\vep, z) \over t}\right).
\end{equation}
Now combining the inequalities \eqref{eq subhpres2} and \eqref{eq subhpres3} with \eqref{eq subhpres1} we obtain
\[
2 u(z, t)-\sup_{B_r(z)}u(\cdot, t)-\inf_{B_r(z)}u(\cdot, t)
\le 2 u_0(c_\vep)-\sup_{B_r(c_\vep)}u_0-\inf_{B_r(c_\vep)}u_0+\vep
\]
for all $r \ge 0$.
Since $u_0$ is uniformly $\infty$-subharmonious with respect to $\delta$, we see that the right hand side is bounded from above by $\vep$ for any $r \leq \delta$.
We conclude the proof by sending $\vep \to 0$.
\end{proof}

Following the same argument, we may also show that (HJ2) preserves $\infty$-subharmoniousness instead of uniform $\infty$-subharmoniousness
provided that supremum in \eqref{hopf-lax2} can be replaced by maximum
because we can assume $c_\vep$ is independent of $\vep$.


For example, an additional assumption on the local compactness of $\X$ makes this work, as stated in the following theorem.

\begin{thm}[Preservation of $\infty$-subharmoniousness]
\label{thm convexity4}
Let $(\X, d)$ be a locally compact geodesic space.
Assume that $H$ satisfies (H1)--(H3).
Let $u_0$ be Lipschitz continuous in $\X$ and $u$ be the unique metric viscosity solution of (HJ2).
Assume that $u_0$ is $\infty$-subharmonious. Then $u(\cdot, t)$ is also $\infty$-subharmonious for all $t \ge 0$.
Moreover, for any $t\geq 0$, $u(\cdot , t)$ is pointwise convex.
\end{thm}

The last statement is an immediate consequence of Proposition \ref{prop har-pt}.
\begin{rmk}
For the same reason as explained in Remark \ref{rmk linear growth}, the preservation of $\infty$-subharmoniousness  holds for more general Hamiltonian with linear growth such as $H(p) = p$ for every $p\geq 0$. 
Indeed, when $H(p) = p$, since, due to Theorem \ref{thm hopf-lax1}, the metric viscosity solution is represented as 
\[
u(x, t) = \sup_{B_t(x)}u_0 = \sup_{a \in \X}\left\{ u_0(a)-t L\left({d(a, x) \over t}\right) \right\}
\]
with $L(v) = \infty$ for $v > 1$, 
following the proof of Theorem \ref{thm convexity3}, one may show that the same results in Theorem \ref{thm convexity3} and Theorem \ref{thm convexity4} hold in this case as well.
\end{rmk}

We finally remark that pointwise convexity itself is not preserved by the Eikonal equation \eqref{c} or the general evolution  \eqref{eq negative hj}.
In fact, the infinity-subharmoniousness of $u_0$ assumed in Theorem \ref{thm convexity4} cannot be replaced by pointwise convexity.
Consider the metric space $\X=\X_1\cup \X_2$ with $\X_1=\{0\}\times \R$, $\X_2=\Rp\times \{0\}$, and let $u_0(x)=0$ for $x\in \X_1$ and $u_0(x)=-x_1$ for $x\in \X_2$.
One may easily verify that $\X^\circ=\X$ and $u_0$ is a pointwise convex function.
However, the solution of \eqref{c} with our choice of $u_0$ is 
\[
u(x, t)=\max_{y\in B_t(x)}u_0(y)=\min\{t-x_1, 0\},
\]
which does not satisfy the pointwise convexity at $(t, 0)\in \X$.
The main reason for the loss of pointwise convexity of $u(\cdot, t)$ is that $u_0$ is not $\infty$-subharmonious;
when $z=(0, 0)$, 
\[
\max_{B_r(z)} u_0+\min_{B_r(z)} u_0=-r<0=u_0(z)
\]
for any $r>0$.

\section{Remarks on convexity preserving properties in discrete spaces}\label{sec lattice}
\subsection{2-dimensional lattices}
In this section, we discuss a particular metric space, the $2$-dimensional lattice graph. Let $\X=\L^2\subset \mathbb{R}^2$ be the square lattice (grid) graph in the plane. Define $\|x\|=|x_1|+|x_2|$ for any $x\in \L^2$, where $(x_1, x_2)$ are the coordinates of $x$, viewed as a point in $\mathbb{R}^2$. The metric on $\L^2$ is given by
\[
d(x, y)=\|x-y\| \quad \text{ for all $x, y\in \L^2$}.
\]
It is easy to verify that $d$ does define a metric on $\L^2$. 

This space $\L^2$ is clearly locally compact. It is not a Busemann space; in other words, the Busemann condition \eqref{eq busemann ineq} does not hold in general. For example, let $x=(0, 0)$, $y=(0, 2k)$ and $y'=(k, 2k)$ with $k\in \mathbb{Z}$. Then it is clear that $z'=(k, k/2)\in M(x', y')$ and $z=(0, 1)$ is the only element in $M(x, y)$. Hence, we have $d(z, z')=3|k|/2$ but $d(x, x')+d(y, y')=|k|$.

However, $\L^2$ equipped with the metric $d$ is a uniform Busemann NPC space. 

\begin{prop}
The metric space $\L^2$ is a uniform Busemann NPC space.
\end{prop}
\begin{proof}
It turns out that one only needs to take $\delta=1/3$. For any $x, y, x', y'\in \L^2$ with $d(x, y)<1/3$, $d(x', x)\leq 1/3$ and $d(y', y)\leq 1/3$, since $d(x', y')<1$, there is only one geodesic joining $x'$ and $y'$. Denote by $\overline{x' y'}$ the unique geodesic from $x'$ to $y'$ and by $\overline{x y}$ the unique geodesic from $x$ to $y$. We discuss two cases below.

Case 1. If $\overline{x' y'}$ and $\overline{x y}$ lie on the same curve in $\L^2$, then this situation can be reduced to a one-dimensional problem; namely, by parametrizing the curve by the arc-length as $\gamma: [0, \tau] \to \L^2$ for some $\tau\leq 1$, we may assume $x=\gamma(a), y=\gamma(b)$ while $x'=\gamma(a')$, $y'=\gamma(b')$ for $a, b, a', b'\in [0, \tau]$. Then the distance between the midpoint of $\overline{x y}$ and that of $\overline{x' y'}$ is 
\[
d\left(\gamma\left(a+b\over 2\right), \gamma\left(a'+b'\over 2\right)\right),
\]
which satisfies the estimate
\[
\begin{aligned}
d\left(\gamma\left(a+b\over 2\right), \gamma\left(a'+b'\over 2\right)\right)&\leq \left|{a+b\over 2}-{a'+b'\over 2}\right|\\
&\leq {1\over 2}|a-a'|+{1\over 2}|b-b'|={1\over 2}d(x, x')+{1\over 2}d(y, y').
\end{aligned}
\]

Case 2. Suppose $\overline{x' y'}$ and $\overline{x y}$ lie on the distinct curves in $\L^2$. Then they must intersect at a certain junction point in the lattice. Without loss, we assume the intersection is the origin $O = (0, 0)$ of the lattice. In this case, we have the following relations for the midpoint $z$ on $\overline{x y}$ and the midpoint $z'$ on $\overline{x' y'}$:
\[
d(z, O)\leq {1\over 2}d(x, O)+{1\over 2}d(y, O),
\quad d(z', O)\leq {1\over 2}d(x', O)+{1\over 2}d(y', O). 
\]
As a result, we have 
\[
d(z, z')=d(z, O)+d(z', O)\leq {1\over 2}d(x, O)+{1\over 2}d(x', O)+{1\over 2}d(y, O)+{1\over 2}d(y', O).
\]
We end our verification by noticing that
\[
d(x, O)+d(x', O)=d(x, x') \quad \text{and}\quad d(y, O)+d(y', O)=d(y, y') .
\] 
\end{proof}
Hence, as a consequence of Theorem \ref{thm convexity1}, the preservation of weak geodesic convexity  holds for the solution of (HJ) in $\L^2$. However, we must point out that the convexity preserving property in this case unfortunately can only apply to trivial cases, since the only weakly geodesically convex functions are constants, as discussed in detail below.



\subsection{Weak geodesic convexity on $\L^2$}

Let us consider the function $f(x)=\|x\|$ on $\L^2$. This function is not geodesically convex either in the weak sense or in the strong sense. For example, when $x=(1/2, 1)$ and $y=(1, 1/2)$, then the midpoint $z=(1, 1)$, but these three points fail to fulfill the convexity condition in Definition \ref{def weak convex}. In fact, the only weakly geodesically convex functions on $\L^2$ are constants, as shown below. 

\begin{prop}\label{trivial convex}
If $u\in C(\L^2)$ is weakly geodesically convex, then $u$ is a constant. 
\end{prop}

\begin{proof}
Let us focus on only one cell, the square $S_0$ with vertices $(0, 0)$, $(1, 0)$, $(1, 1)$ and $(0, 1)$. 
Without loss, we may assume that $u(0, 0)=0$. Since there is only one geodesic joining $(\vep, 0)$ and $(1, 1-\vep)$ for $\vep>0$ small, the weak geodesic convexity of $u$ yields
\[
u(1, 1-\vep)+u(\vep, 0)\geq 2u(1, 0). 
\]
Sending $\vep\to 0$, we have
\begin{equation}\label{lattice convex1}
u(1, 1)\geq 2u(1, 0)
\end{equation}
due to the continuity of $u$. Similarly, we can obtain
\begin{equation}\label{lattice convex2}
u(1, 1)\geq 2 u(0, 1)
\end{equation} 
and 
\begin{equation}\label{lattice convex3}
u(0, 1)+u(1, 0)\geq 2 u(1, 1).
\end{equation} 
Combining \eqref{lattice convex3} with \eqref{lattice convex1} and \eqref{lattice convex2} respectively, we have
\[
u(0, 1)\geq 3 u(1, 0)
\]
and
\[
u(1, 0)\geq 3 u(0, 1),
\]
which implies that $u(0, 1)=u(1, 0)=0$ and therefore $u\equiv 0$ on $S_0$. We may extend this argument to the neighboring cells and eventually show that $u\equiv 0 = u(0, 0)$ on $\L^2$. 
\end{proof}

Such a result shows that Theorem \ref{thm convexity1} only applies to the trivial cases when $u\equiv u_0$ is constant in $\L^2$. In this discrete setting, it seems that we need to further relax the notion of convexity, using the following definition. 
\begin{defi}\label{def lattice convex}
We say a function $u\in C(\L^2)$ is 1-weakly geodesically convex if $u$ satisfies \eqref{eq weak convex} (with ``$\inf$'' replaced by ``$\min$'') but only for all $x=(x_1, x_2), y=(y_1, y_2)\in \X=\L^2$ satisfying either 
\begin{equation}\label{lattice constraint0}
\min\{|x_1-y_1|, |x_2-y_2|\}=0
\end{equation}
or
\begin{equation}\label{lattice constraint}
\min\{|x_1-y_1|, |x_2-y_2|\}\geq 1.
\end{equation}
\end{defi}

\begin{prop}
The norm function $f(x)=\|x\|$ on $\L^2$ is $1$-weakly geodesically convex but not $1$-strongly geodesically convex.
\end{prop} 
\begin{proof}
It is clear that $f$ satisfies the convexity inequality under the constraint \eqref{lattice constraint0}. It suffices to give the proof in the case \eqref{lattice constraint}. 

For any $x=(x_1, x_2), y=(y_1, y_2)$ from $\L^2$, assuming first that $x_1+y_1\geq 0$ and $x_2+y_2\geq 0$ with $x_2+y_2\geq x_1+y_1$, we fix a point $z$ in the following way:

\noindent {Case 1.} When $(x_1-y_1)(x_2-y_2)\geq0$, we take
\[
z=\begin{cases}
\big(\left[{(x_1+y_1)/ 2}\right], {(x_2+y_2)/ 2}+e_1\big), &\text{ if $e_1\leq {1\over 2}$,}\\
\big(\left[{(x_1+y_1)/ 2}\right]+1, {(x_2+y_2)/ 2}-1+e_1\big), &\text{ if $e_1>{1\over 2}$,}
\end{cases}
\]
where 
\begin{equation}\label{e1}
e_1={x_1+y_1\over 2}-\left[{x_1+y_1\over 2}\right].
\end{equation}
A direct verification with an application of \eqref{lattice constraint} yields that
\[
\begin{aligned}
d(x, z)&=\begin{cases}
\left|(y_1-x_1)/2-e_1\right|+\left|(y_2-x_2)/2+e_1\right|, &\text{if $e_1\leq {1/ 2}$,}\\
\left|(y_1-x_1)/2+1-e_1\right|+\left|(y_2-x_2)/2-1+e_1\right|, &\text{if $e_1>{1/ 2}$}
\end{cases}\\
&= {1\over 2}|x_1-y_1|+{1\over 2}|x_2-y_2|={1\over 2}d(x, y). 
\end{aligned}
\]
We can similarly get
\[
d(y, z)={|x_1-y_1|\over2}+{|x_2-y_2|\over 2}={1\over 2}d(x, y).
\]
Hence, we deduce $z\in M(x, y)$ in this case.

{\noindent Case 2.} When $(x_1-y_1)(x_2- y_2)<0$, we let
\begin{equation}\label{eq:lattice3}
z=\begin{cases}
\big(\left[{(x_1+y_1)/ 2}\right], {(x_2+y_2)/ 2}-e_1\big), &\text{ if $e_1\leq e_2$,}\\
\big({(x_1+y_1)/ 2}-e_2, \left[{(x_2+y_2)/ 2}\right]\big), &\text{ if $e_1>e_2$,}
\end{cases}
\end{equation}
where $e_1$ is given by \eqref{e1} and 
\[
e_2={x_2+y_2\over 2}-\left[{x_2+y_2\over 2}\right].
\]

We can still show that $z\in M(x, y)$ in this case. We only give a proof for the case that $e_1\leq e_2$. The case that $e_1>e_2$ can be handled analogously. Note that when $e_1\leq e_2$, our choice of $z$ yields
\[
d(x, z)=\left|(y_1-x_1)/2-e_1\right|+\left|(y_2-x_2)/2-e_1\right|, 
\]
which implies that 
\[
d(x, z)= {1\over 2}|x_1-y_1|+{1\over 2}|x_2-y_2|={1\over 2}d(x, y)
\]
due to \eqref{lattice constraint}. We can show in a similar way that $d(y, z)=d(x, y)/2$ as well. 

We next show that the convexity inequality holds for $f$ with such a choice of $z$. We again discuss the two cases above. In Case 1, we have 
\[
f(z)=\left|{(x_1+y_1)/ 2}-k\right|+\left|(x_2+y_2)/ 2+k\right|  \quad \text{for $k=\min\{e_1, 1-e_1\}$}.
\]
Since 
\[
e_1\leq {x_1+y_1\over 2}\leq {x_2+y_2\over 2},
\]
it follows that 
\[
2f(z)=x_1+y_1+x_2+y_2\leq |x_1|+|y_1|+|x_2|+|y_2|=f(x)+f(y).
\]
The situation in Case 2 is simpler. By \eqref{eq:lattice3}, we get  
\[
2f(z)\leq |x_1+y_1|+|x_2+y_2|\leq |x_1|+|y_1|+|x_2|+|y_2|=f(x)+f(y).
\]  
If $x_2+y_2\leq x_1+y_1$ holds, we only need to switch the roles of $\{x_1, y_1\}$ and $\{x_2, y_2\}$ in the argument above. In general, when the inequalities $x_1+y_1\geq 0$ and $x_2+y_2\geq 0$ fail to hold, we are still able to find the midpoint $z$ as above by symmetry.

We next show that $u$ is not 1-strongly geodesically convex in $\L^2$. Take $x=(1, 0)$ and $y=(0, 1)$. It is clear that $z=(1, 1)$ is a midpoint of $x$ and $y$ but $u(x)=u(y)=1$ and $u(z)=2$.
\end{proof}

In spite of the relaxation of convexity notions on functions in $\L^2$, the solution of (HJ) still fails to preserve (1-weak geodesic) convexity in general. 
\begin{prop}[Non-preservation of 1-weak geodesic convexity in $\L^2$]
Let $\X=\L^2$ and $u(x, t)$ be the metric viscosity solution of the time-dependent Eikonal equation \eqref{c} with initial condition $u(\cdot , 0)=u_0$ in $\X$. Then there exists a 1-weakly geodesically convex, Lipschitz function $u_0$ on $\L^2$ such that the solution $u$ fails to be 1-weakly geodesically convex in $x$ at some $t>0$. 
\end{prop}

\begin{proof}
Let $R>0$ be sufficiently large. For any $x=(x_1, x_2)\in B_R(0, 0)\subset \L^2$, let
\[
u_0(x)=
\begin{cases}
(x_1+1)x_2 & \text{if $x_1\geq 0$ and $x_2\geq 0$},\\
0 &\text{otherwise.}
\end{cases}
\]
It is not difficult to see that $u_0$ is 1-weakly geodesically convex in $B_R(0, 0)$.
We extend $u_0$ outside $B_R(0, 0)$ so that $u_0$ is Lipschitz and 1-weakly geodesically convex in $\L^2$. 

Recall that by the optimal control interpretation the solution $u$ is given by \eqref{eq hl1}.
In particular, taking $x=(k+1, k)$, $y=(k, 3k)$ for a positive integer $k\geq 4$ and letting $R\geq 5k$, we have 
\[
u(x, k)=0, \quad u(y, k)\leq u_0((0, 3k))=3k.
\]
On the other hand, we consider the value $u(z, k)$ for all $z\in M(x, y)$. There are three choices of $z$: (a) $z=(k, 2k-1/2)$, (b) $z=(k+1/2, 2k)$ and (c) $z=(k+1, 2k+1/2)$.
By using \eqref{eq hl1} again, we have
\[
u(z, k)
= \begin{cases}
3k-3/4 & \text{if $z=\left(k, 2k-{1/2}\right)$},\\
3k &\text{if $z=\left(k+1/2, 2k\right)$},\\
5k &\text{if $z=(k+1, 2k+1/2)$}.
\end{cases}
\]
In any case, we get
\[
u(x, k)+u(y, k)-2u(z, k)<-2k,
\]
which indicates the failure of preservation of 1-weak geodesic convexity. 
\end{proof}

We obtain an affirmative but weaker convexity result if we further relax the notion of convexity by adopting Definitions \ref{def wp-convex} and \ref{def point-convex}. In view of Theorem \ref{thm convexity4}, if we change the sign of the Hamiltonian $H$ and assume that $u_0$ is infinity-subharmonious in $\L^2$, then the solution $u$ is also infinity-subharmonious in $\L^2$ for all time $t\geq 0$. Thanks to Proposition \ref{prop har-pt} it follows that $u(\cdot, t)$ is pointwise convex, as defined in Definition \ref{def u wp-convex}, for all time $t\geq 0$.

We finally remark that Theorem \ref{thm convexity4}, applied to the current case in $\L^2$,  does not only yield trivial (constant preserving) situations. Indeed, it is easily seen that $x\mapsto \|x\|$ is infinity-subharmonious and pointwise convex in $\L^2$.

\section*{Appendix}

We give the proof of Theorem \ref{thm hopf-lax1}, which is based on an approximation of $H(p) = p$.

\begin{proof}[Proof of Theorem \ref{thm hopf-lax1}]
Consider the family of functions $p \mapsto p^\alpha/\alpha$ with $\alpha \in (1, \infty)$.
Note that it converges to $p$ locally uniformly as $\alpha \to 1$.
Since, for each $\alpha$, $H(p) = p^\alpha/\alpha$ satisfies the conditions (H1)--(H3),
we see by Theorem \ref{thm solhj} that the unique solution of \eqref{hj1} with $H(p) = p^\alpha/\alpha$ is given by
\[
u^\alpha(x, t) := \inf_{a \in \X}\left\{ u_0(a)+{t \over \beta}\left({d(a, x) \over t}\right)^\beta\right\},
\]
where $\beta := \alpha/(\alpha-1) \in (1, \infty)$.

Let us estimate
\[
u^\alpha(x, t)-u(x, t)
= \inf_{a \in \X}\left\{ u_0(a)+{t \over \beta}\left({d(a, x) \over t}\right)^\beta\right\}-\inf_{b \in B_t(x)}u_0(b)
\]
for $\alpha > 1$, $t \in (0, \infty)$, $x \in X$.
For simplicity, we will write $\vep := \alpha-1$.

The upper bound estimate is easier:
Note that there is $b \in B_t(x)$ such that $u_0(b) < u(x, t)+\vep$.
Taking $a = b$ we have
\[
u^\alpha(x, t)-u(x, t)
\le {t \over \beta}\left({d(a, x) \over t}\right)^\beta+\vep
\le t/\beta+\vep
= (\alpha-1)t/\alpha+\vep.
\]

To obtain the lower bound estimate, we apply Proposition \ref{thm finiteprop} and obtain \eqref{eq hl2} with
\[
V = V_\alpha := (\beta K)^{1 \over \beta-1} = \left({\alpha \over \alpha-1}K\right)^{\alpha-1},
\]
where $K$ is the Lipschitz constant of $u_0$.
Hence, there is $a \in B_{V_\alpha t}(x)$ such that
\[
u_0(a)+{t \over \beta}\left({d(a, x) \over t}\right)^\beta \le u^\alpha(x, t)+\vep.
\]
Thanks to the existence of a geodesic between $x$ and $a$,
we can take $b \in B_t(x)$ such that
$
d(b, a) \le (V_\alpha-1)t
$
and then
\[
u^\alpha(x, t)-u(x, t)
\ge u_0(a)+{t \over \beta}\left({d(a, x) \over t}\right)^\beta-u_0(b)-\vep
\ge -K (V_\alpha-1)t-\vep.
\]
Note that the right hand side converges to $0$ locally uniformly as $\alpha \to 1$ because $V_\alpha \to 1$ by the definition.

In conclusion, we see that $u^\alpha$ converges to $u$ locally uniformly on $[0, \infty)\times X$.
Since $p^\alpha/\alpha$ converges to $p$ locally uniformly on $\Rp$,
the stability property (\cite{NN}) implies that $u$ is a metric viscosity solution of \eqref{c}.
\end{proof}

\bibliographystyle{abbrv}%
\bibliography{bib_liu}%
\end{document}